\documentclass[12pt]{article}
\usepackage[top=2.5cm,bottom=2.5cm,left=2.9cm,right=2.9cm]{geometry}
\usepackage{amssymb}
\usepackage{amsmath,amsthm}
\usepackage{enumerate}
\usepackage{tikz}
\usepackage{verbatim}
\usepackage{upgreek}
 \usepackage{mathptmx}
\usetikzlibrary{calc}

\newtheorem{remark}{Remark}[section]

\newtheorem{lemma}[remark]{Lemma}
\newtheorem{theorem}[remark]{Theorem}
\newtheorem{proposition}[remark]{Proposition}

\newtheorem{corollary}[remark]{Corollary}

\title{From (secure) $w$-domination in graphs to protection of lexicographic product graphs}
\author{A. Cabrera Mart\'{\i}nez, A. Estrada-Moreno, J. A. Rodr\'{\i}guez-Vel\'{a}zquez\\
 {\small Universitat Rovira i Virgili }\\{\small Departament d'Enginyeria Inform\`atica i Matem\`atiques } \\  {\small Av. Pa\"{\i}sos Catalans 26, 43007 Tarragona, Spain.} \\{\small
  abel.cabrera\@@urv.cat, alejandro.estrada\@@urv.cat, juanalberto.rodriguez\@@urv.cat}\\
}

\date{ }
\begin{document}
\maketitle

\begin{abstract}
 Let  $w=(w_0,w_1, \dots,w_l)$ be a vector of nonnegative integers such that $ w_0\ge 1$. Let $G$ be a graph and $N(v)$  the open neighbourhood of $v\in V(G)$.  We say that a function $f: V(G)\longrightarrow \{0,1,\dots ,l\}$ is a  $w$-dominating function if  $f(N(v))=\sum_{u\in N(v)}f(u)\ge w_i$ for every vertex $v$ with $f(v)=i$. The weight of $f$ is defined to be $\omega(f)=\sum_{v\in V(G)} f(v)$. 
Given a $w$-dominating function $f$ and any pair of adjacent vertices $v, u\in V(G)$ with $f(v)=0$ and $f(u)>0$, the  function $f_{u\rightarrow v}$ is defined by $f_{u\rightarrow v}(v)=1$, $f_{u\rightarrow v}(u)=f(u)-1$ and $f_{u\rightarrow v}(x)=f(x)$ for every $x\in V(G)\setminus\{u,v\}$. 
We say that a  $w$-dominating function $f$ is a secure  $w$-dominating function if   for every $v$ with $f(v)=0$,  there exists $u\in N(v)$ such that $f(u)>0$ and  $f_{u\rightarrow v}$ is a  $w$-dominating function as well. The  (secure) $w$-domination number of $G$, denoted by ($\gamma_{w}^s(G)$) $\gamma_{w}(G)$, is defined as the minimum weight among all (secure) $w$-dominating functions. \\
  In this paper,  we show how
the secure (total) domination number and the (total) weak  Roman domination number of lexicographic product graphs $G\circ H$ are related to $\gamma_w^s(G)$ or $\gamma_w(G)$. For the case of the secure domination number and the weak Roman domination number, the decision on whether $w$ takes specific components will depend on the value of $\gamma_{(1,0)}^s(H)$,   while in the case of the total version of these parameters, the decision will depend on the value of $\gamma_{(1,1)}^s(H)$. 
\end{abstract}


\vspace{0,5cm}

\section{Introduction}
As usual,  $\mathbb{Z}^+=\{1,2,3,\dots\}$ and $\mathbb{N}=\mathbb{Z}^+\cup\{0\}$ denote the sets  of positive and nonnegative integers, respectively.
Let $G$ be a graph,  $l\in \mathbb{Z}^+$ an integer, and $f: V(G)\longrightarrow \{0,\dots , l\}$ a function. Let $V_i=\{v\in V(G):\; f(v)=i\}$ for every $i\in \{0,\dots , l\}$. We will identify $f$ with the subsets $V_0,\dots,V_l$ associated with it, and so we will use the unified notation $f(V_0,\dots , V_l)$ for the function and these associated subsets.  The \emph{weight} of $f$ is defined as $$\omega(f)=f(V(G))=\sum_{i=1}^l i|V_i| .$$

Let $w=(w_0, \dots,w_l)\in \mathbb{Z}^+\times \mathbb{N}^l$ such that $ w_0\ge 1$. 
As defined in \cite{w-domination}, a function $f(V_0,\dots , V_l)$ is a  $w$-\emph{dominating function} if  $f(N(v))\geq w_i$ for every $v\in V_i$. 
The  $w$-\emph{domination number} of $G$, denoted by $\gamma_{w}(G)$, is the minimum weight among all $w$-dominating functions. For simplicity, a  $w$-dominating function $f$  of weight $\omega(f)=\gamma_{w}(G)$ will be called a $\gamma_{w}(G)$-function. For fundamental results on the $w$-domination number of a graph, we refer the interested readers to \cite{w-domination}; the paper where the theory of  $w$-domination in graphs was introduced.

%
 
For any function $f(V_0,\dots , V_l)$ and any pair of adjacent vertices $v\in V_0$ and $u\in V(G)\setminus V_0$, the  function $f_{u\rightarrow v}$ is defined by $f_{u\rightarrow v}(v)=1$, $f_{u\rightarrow v}(u)=f(u)-1$ and $f_{u\rightarrow v}(x)=f(x)$ whenever $x\in V(G)\setminus\{u,v\}$. 

The authors of this paper \cite{sym12121948} introduced the approach of secure $w$-domination as follows.
A  $w$-dominating function $f(V_0,\dots , V_l)$ is a \emph{secure}  $w$-\emph{dominating function} if   for every $v\in V_0$ there exists $u\in N(v)\setminus V_0$ such that $f_{u\rightarrow v}$ is a  $w$-dominating function as well. The  \emph{secure} $w$-\emph{domination number} of $G$, denoted by $\gamma_{w}^s(G)$, is the minimum weight among all secure $w$-dominating functions. For simplicity, a secure  $w$-dominating function $f$  of weight $\omega(f)=\gamma_{w}^s(G)$ will be called a $\gamma_{w}^s(G)$-function. 
This approach  to the theory of secure domination covers the different versions of secure domination known so far. For instance,
we would emphasize the following cases of known parameters that we define here in terms of secure $w$-domination.

\begin{itemize}
\item The \emph{secure domination number} of $G$ is defined to be $\gamma_s(G)=\gamma_{(1,0)}^s(G)$. In this case, for any secure  $(1,0)$-dominating function $f(V_0,V_1)$, the set  $V_1$ is known as a \emph{secure dominating set}. This concept  was introduced  by Cockayne et al. in \cite{MR2137919}  and studied further in several papers, including among others, \cite{MR3355313,MR2529132,
MR3258160,Cockayne2003,MR2477230,Valveny2018}.

\item The \emph{secure total domination number} of a graph $G$ of minimum degree at least one  is defined to be $\gamma_{st}(G)=\gamma_{(1,1)}^s(G)$. In this case, for any  secure $(1,1)$-dominating function $f(V_0,V_1)$, the set $V_1$ is known as a \emph{secure total dominating set} of $G$.
This concept was introduced  by Benecke et al.\ in \cite{MR2364004} and studied further in several papers, including among others,  \cite{STDS2019,TWRDF2018,MR3624798,MR2477230,MR3773180}.

\item The \emph{weak Roman domination number} of a graph  $G$ is defined to be  $\gamma_{r}(G)=\gamma_{(1,0,0)}^s(G)$. 
 This concept was introduced by Henning and Hedetniemi \cite{MR1991720} and studied further in several papers, including among others, \cite{Cabrera-weak-tree-2020,MR3258160,Cockayne2003,Valveny2017}.

\item The \emph{total weak Roman domination number} of a graph  $G$ of minimum degree at least one is defined to be $\gamma_{tr}(G)=\gamma_{(1,1,1)}^s(G)$. This concept was introduced by Cabrera et al.\  in \cite{TWRDF2018} and studied further in \cite{TPlexicographic-2019}.

\item The \emph{secure Italian domination number} of $G$ is defined to be $\gamma_{_I}^s(G)=\gamma_{(2,0,0)}^s(G)$. This parameter was introduced by Dettlaff et al.\ in \cite{SecureItalian}.
\end{itemize}


 In this paper  we show how
the secure (total) domination number and the (total) weak  Roman domination number of lexicographic product graphs $G\circ H$ are related to $\gamma_w^s(G)$ or $\gamma_w(G)$. For the case of the secure domination number and the weak Roman domination number, the decision on whether $w$ takes specific components will depend on the value of $\gamma_{(1,0)}^s(H)$,   while in the case of the total version of these parameters, the decision will depend on the value of $\gamma_{(1,1)}^s(H)$.

We assume that the reader is familiar with the basic concepts, notation  and terminology of domination in graph. If this is not the case, we suggest the textbooks \cite{Haynes1998,Haynes1998a}.  For the remainder of the paper, definitions will be introduced whenever a concept is needed.

\section{Some tools}\label{NewSecureDomination}

 Given a  $w$-dominating function $f(V_0,\dots , V_l)$ and $v\in V_0$, we define $$M_f(v)=\{u\in V(G)\setminus V_0: \, f_{u\rightarrow v} \text{ is a } w\text{-dominating function}\}.$$
Obviously, if $f$ is a secure  $w$-dominating function, then $M_f(v)\ne \varnothing$ for every $v\in V_0$.

\begin{theorem}{\rm \cite{sym12121948}}\label{PreliminaryBounds-w-Secure}
Let $G$ be a graph of minimum degree $\delta$, and let $w=(w_0,\dots ,w_l)\in \mathbb{Z}^+\times \mathbb{N}^l$ such that  $w_i\ge w_{i+1}$ for every $i\in \{0, \dots , l-1\}$. If $l\delta\ge w_l$, then following statements hold.

\begin{enumerate}
\item[{\rm (i)}] $\gamma_{w}(G)\le \gamma_{w}^s(G).$
\\
\item[{\rm (ii)}] If $k\in \mathbb{Z}^+$, then $\gamma_{(k+1,k=w_1,\dots, w_l)}(G)\leq \gamma_{(k,k=w_1,\dots, w_l)}^s(G)$.
\end{enumerate}
\end{theorem}

\begin{theorem}{\rm \cite{sym12121948}}\label{PreliminaryBounds-w-SecureSegundTh}
Let $G$ be a graph of minimum degree $\delta$, and let $w=(w_0,\dots ,w_l), w'=(w_0',\dots ,w'_l)\in \mathbb{Z}^+\times \mathbb{N}^l$ such that $l\delta\ge w_l$,  $ w_i\ge w_{i+1}$ and  $w_i'\ge w_{i+1}'$ for every $i\in \{0, \dots , l-1\}$.
If $w_i\ge w_{i-1}'-1$ for every $i\in \{1,\dots ,l\}$, and $\max\{w_{j}-1,0\}\ge w_j'\,$  for every $j\in \{0,\dots ,l\}$, then $$\gamma_{w'}^s(G)\le \gamma_{w}(G).$$
\end{theorem}

The following result is a particular case of Theorem \ref{PreliminaryBounds-w-SecureSegundTh}.

\begin{corollary}{\rm \cite{sym12121948}}\label{CorollarySecureSumaUno}
Let $G$ be a graph of minimum degree $\delta$, and let $w=(w_0,\dots ,w_l)\in \mathbb{Z}^+\times \mathbb{N}^l$ and $\textbf{1}=(1,\dots,1)$. If $ 0\le w_{j-1}-w_j\le 2$ for every $j\in \{1,\dots ,i\}$, where $1\le i\le l$  and $l\delta\ge w_l+1$, then $$\gamma_{(w_0,\dots,w_i,0, \dots ,0 )}^s(G)\le \gamma_{(w_0+1,\dots,w_i+1,0, \dots ,0 )}(G)\le \gamma_{w+\textbf{1}}(G).$$
\end{corollary}

\begin{proposition}{\rm \cite{sym12121948}}\label{obs-subgraph-general-Secure}
Let $G$ be a graph of order $n$. Let
$w=(w_0,\dots ,w_l)\in \mathbb{Z}^+\times \mathbb{N}^l$ such that $ w_0\ge  \cdots \ge w_l$.
If $G'$ is a spanning subgraph of $G$ with 
 minimum degree $\delta'\ge \frac{w_l }{l}$,   then $$\gamma_{w}^s(G)\leq \gamma_{w}^s(G').$$
\end{proposition}

\section{The case of lexicographic product graphs}\label{SectionlexicographicSecure}

The \emph{lexicographic product} of two graphs $G$ and $H$ is the graph $G \circ H$ whose vertex set is  $V(G \circ H)=  V(G)  \times V(H )$ and $(u,v)(x,y) \in E(G \circ H)$ if and only if $ux \in E(G)$ or $u=x$ and $vy \in E(H)$.

Notice that  for any $u\in V(G)$  the subgraph of $G\circ H$ induced by $\{u\}\times V(H)$ is isomorphic to $H$. For simplicity, we will denote this subgraph by $H_u$. Moreover,
 the neighbourhood of $(x,y)\in V(G)\times V(H)$ will be denoted by $N(x,y)$ instead of $N((x,y))$. Analogously, for any function $f$ on $G\circ H$, the image of $(x,y)$   will be denoted by $f(x,y)$ instead of $f((x,y))$.

The next subsections are devoted to  show how
the secure (total) domination number and the (total) weak  Roman domination number of lexicographic product graphs $G\circ H$ are related to $\gamma_w^s(G)$ or $\gamma_w(G)$, for certain vectors $w$ of three components.

\subsection{Secure  domination}

\begin{lemma}\label{lem-vertice<=2-secure}
For any  graph $G$ with no isolated vertex and any nontrivial graph $H$, there exists a $\gamma_{(1,0)}^s(G\circ H)$-function $f$ such that $f(V(H_u))\leq 2$  for every $u\in V(G)$.
\end{lemma}

\begin{proof}
Given a secure $(1,0)$-dominating function $f$ on $G\circ H$, we define $$R_f=\{x\in V(G): \, f(V(H_x))\ge 3\}.$$ Let $f$ be a $\gamma_{(1,0)}^s(G\circ H)$-function such that $|R_f|$ is minimum among all $\gamma_{(1,0)}^s(G\circ H)$-functions.

Suppose that $|R_f|\geq 1$.  Let $u\in R_f$,  $u'\in N(u)$ and  $v_1,v_2\in V(H)$ such that $f(u,v_1)=f(u,v_2)=1$.  Now, let
 $f':V(G)\times V(H)\longrightarrow \{0,1\}$  be a function defined as follows.
\begin{itemize}
\item[$\bullet$] $f'(u,v_1)=f'(u,v_2)=1$ and $f'(u,y)=0$ for every $y\in V(H)\setminus \{v_1,v_2\}$;
\item[$\bullet$]  $f'(V(H_{u'}))=\min\{2,f(V(H_{u'}))+f(V(H_{u}))-2\}$;
\item[$\bullet$] $f'(x,y)=f(x,y)$ for every $x\in V(G)\setminus \{u,u'\}$ and $y\in V(H)$.
\end{itemize}
 It is not difficult to check that $f'$ is a secure $(1,0)$-dominating function on $G\circ H$ with $\omega(f')\le \omega(f)$ and $|R_{f'}|<|R_f|$, which is a contradiction.  Therefore, $R_f=\varnothing $, and the result follows.
\end{proof}  
  
We shall need the following two results.

\begin{theorem}\label{teo-s=r-lexicographic}{\rm \cite{SecureLexicographicDMGT}} 
For any graph $G$ with no isolated vertex and any nontrivial graph $H$ with $\gamma_{(1,0)}^s(H)\leq 2$ or $\gamma_{(1,0,0)}^s(H)\geq 3$,
$$\gamma_{(1,0)}^s(G\circ H)=\gamma_{(1,0,0)}^s(G\circ H).$$
\end{theorem}  

\begin{proposition}\label{prop-r-Kn-lexicographic}{\rm \cite{Valveny2017}}
For any graph $G$ and any integer $n\geq 1$,
$$\gamma_{(1,0,0)}^s(G\circ K_n)=\gamma_{(1,0,0)}^s(G)$$
\end{proposition}

The following result shows how the secure domination number of $G\circ H$ is related to   $\gamma_{w}^s(G)$ or $\gamma_{w}(G)$ for certain vectors $w$ of three components. The decision on whether  the components of $w$ take specific values will depend on the value of $\gamma_{(1,0)}^s(H)$ and $\gamma(H)$. 

\begin{theorem}\label{teo-equality-gamma(1,0)}
For a graph $G$ with no isolated vertex and a nontrivial graph $H$, the following statements hold.
 
\begin{enumerate}
\item[{\rm (i)}] If $\gamma_{(1,0)}^s(H)=1$, i.e., $H$ is a complete graph, then $\gamma_{(1,0)}^s(G\circ H)=\gamma_{(1,0,0)}^s(G)$. \\
\item[{\rm (ii)}] If $\gamma_{(1,0)}^s(H)=2$ and $\gamma(H)=1$, then $\gamma_{(1,0)}^s(G\circ H)=\gamma_{(2,1,0)}(G)$. \\
\item[{\rm (iii)}] If $\gamma_{(1,0)}^s(H)\ge 3$ and $\gamma(H)=1$, then $\gamma_{(1,0)}^s(G\circ H)=\gamma_{(2,1,1)}(G)$. \\
\item[{\rm (iv)}] If $\gamma_{(1,0)}^s(H)=\gamma(H)=2$, then $\gamma_{(1,1,0)}^s(G)\le \gamma_{(1,0)}^s(G\circ H)\le\gamma_{(2,2,0)}(G)$. \\
\item[{\rm (v)}] If $\gamma_{(1,0)}^s(H)>\gamma(H)=2$, then $\gamma_{(1,0)}^s(G\circ H)=\gamma_{(2,2,1)}(G)$. \\
\item[{\rm (vi)}] If  $\gamma_{(1,0)}^s(H)=\gamma(H)=3$, then $\gamma_{(2,2,1)}(G)\le\gamma_{(1,0)}^s(G\circ H)\le \gamma_{(2,2,2)}(G)$.
\\
\item[{\rm (vii)}] If $\gamma_{(1,0)}^s(H)\geq 4$ and $\gamma(H)\geq 3$, then $\gamma_{(1,0)}^s(G\circ H)=\gamma_{(2,2,2)}(G)$.
\end{enumerate}
\end{theorem}

\begin{proof}

Let $f(V_0,V_1)$ be a $\gamma_{(1,0)}^s(G\circ H)$-function which satisfies Lemma \ref{lem-vertice<=2-secure}. Let  $f'(X_0,X_1,X_2)$ be the function defined on $G$ by $X_1=\{x\in V(G): f(V(H_x))=1\}$ and $X_2=\{x\in V(G): f(V(H_x))=2\}$. Notice that $\gamma_{(1,0)}^s(G\circ H)=\omega(f)=\omega(f')$. With this notation in mind, we differentiate the following cases.

\vspace{.2cm}
\noindent
Case $1.$ $\gamma_{(1,0)}^s(H)=1$. In this case, $H$ is a complete graph. Hence, by Theorem \ref{teo-s=r-lexicographic} and Proposition \ref{prop-r-Kn-lexicographic} we deduce that $\gamma_{(1,0)}^s(G\circ H)=\gamma_{(1,0,0)}^s(G\circ H)=\gamma_{(1,0,0)}^s(G)$.

\vspace{.2cm}
\noindent
Case $2.$ $\gamma_{(1,0)}^s(H)= 2$ and $\gamma(H)=1$. In this case, if $x\in X_0$, then  $f'(N(x))=f(N(V(H_x))\setminus V(H_x))\ge 2$. Now, since $\gamma_{(1,0)}^s(H)= 2$,  if $x\in X_1$, then $f'(N(x))=f(N(V(H_x))\setminus V(H_x))\ge 1$. Therefore,  $f'$ is a  $(2,1,0)$-dominating function on $G$, which implies that $\gamma_{(1,0)}^s(G\circ H)=\omega(f)=\omega(f')\ge \gamma_{(2,1,0)}(G)$.

On the other side, 
for any $\gamma_{(2,1,0)}(G)$-function  $g(W_0,W_1,W_2)$ and any universal vertex $v$ of $H$, the function $g'(W_0',W_1',W_2')$, defined by $W_1'=W_1\times \{v\}$ and $W_2'=W_2\times \{v\}$, is a $(2,1,0)$-dominating function on $G\circ H$. 
Hence, $\gamma_{(2,1,0)}(G\circ H)\le \omega(g')=\omega(g)=\gamma_{(2,1,0)}(G)$. Therefore, by  
Theorem \ref{teo-s=r-lexicographic}  and Corollary \ref{CorollarySecureSumaUno} we conclude that $\gamma_{(1,0)}^s(G\circ H)=\gamma_{(1,0,0)}^s(G\circ H)\leq \gamma_{(2,0,0)}(G\circ H)\le  \gamma_{(2,1,0)}(G\circ H)\le \gamma_{(2,1,0)}(G)$.

\vspace{.2cm}
\noindent
Case $3.$ $\gamma_{(1,0)}^s(H)\ge 3$ and $\gamma(H)=1$. As above, if $x\in X_0$, then  $f'(N(x))=f(N(V(H_x))\setminus V(H_x))\ge 2$ and, since $\gamma_{(1,0)}^s(H)\ge 3$, if $x\in X_1\cup X_2$, then  $f'(N(x))=f(N(V(H_x))\setminus V(H_x))\ge 1$. Hence,  $f'$ is a  $(2,1,1)$-dominating function on $G$. Therefore, $\gamma_{(1,0)}^s(G\circ H)=\omega(f)=\omega(f')\ge \gamma_{(2,1,1)}(G)$.

On the other side, 
for any $\gamma_{(2,1,1)}(G)$-function  $g(W_0,W_1,W_2)$, any universal vertex $v$ of $H$ and  any $v'\in V(H)\setminus \{v\}$, the function $g'(W_0',W_1')$, defined by $W_1'= W_1\times \{v\} \cup W_2\times \{v,v'\}$, is a  $(1,0)$-dominating function on $G\circ H$. Now, for any $(x,y)\in W_0'$ with $x\in W_1\cup W_2$, we can see that $g'_{(x,v)\rightarrow (x,y)}$ is a $(1,0)$-dominating function on $G\circ H$, while for any $(x,y)\in W_0\times V(H)$ there exists $x'\in W_2\cap N(x)$ or $x',x''\in W_1\cap N(x)$, and so $g'_{(x',v)\rightarrow (x,y)}$ is a $(1,0)$-dominating function on $G\circ H$. Therefore, $g'$ is 
a secure $(1,0)$-dominating function on $G\circ H$ and, as a consequence,  $\gamma_{(1,0)}(G\circ H)\le \omega(g')=\omega(g)=\gamma_{(2,1,1)}(G)$.


\vspace{.2cm}
\noindent
Case $4.$ $\gamma_{(1,0)}^s(H)=\gamma(H)=2$. If $x\in X_0\cup X_1$, then  $f'(N(x))=f(N(V(H_x))\setminus V(H_x))\ge  1$, which implies that  $f'$ is a  $(1,1,0)$-dominating function on $G$. Now, for any  $(x,y)\in X_0\times  V(H)$,  there exists $(x',y')\in M_f(x,y)$ with $x'\in N(x)\cap(X_1\cup X_2)$. Hence, for any $u\in X_0\cup \{x'\}$ we have that $f'_{x'\rightarrow x}(N(u))=f_{(x',y')\rightarrow (x,y)}(N(V(H_u))\setminus V(H_u))\ge 1$. Now, if $u\in X_1$, then as $\gamma_{(1,0)}^s(H_u)=2$ we have that $f'_{x'\rightarrow x}(N(u))=f_{(x',y')\rightarrow (x,y)}(N(V(H_u))\setminus V(H_u))\ge 1$. Hence, $f'$ is a secure $(1,1,0)$-dominating function on $G$. Therefore, $\gamma_{(1,0)}^s(G\circ H)=\omega(f)=\omega(f')\ge \gamma_{(1,1,0)}^s(G)$.

On the other side,
for any $\gamma_{(2,2,0)}(G)$-function  $g(W_0,W_1,W_2)$ and any $\gamma_{(1,0)}^s(H)$-function $h(Y_0,Y_1)$ with  $Y_1=\{v_1,v_2\}$, the function $g'(W_0',W_1')$, defined by $W_1'=W_1\times \{v_1\}\cup W_2\times Y_1$, is a $(1,0)$-dominating function on $G\circ H$. Now,  for any  $(x,y)\in W_0'\setminus W_2\times V(H)$ and $x'\in N(x)\cap (W_1\cup W_2)$, we can see that $g'_{(x',v_1)\rightarrow (x,y)}$ is a $(1,0)$-dominating function on $G\circ H$.
Furthermore, for every $(x,y)\in W_0'\cap W_2\times V(H)$  there exists $v_i\in \{Y_1\}\cap M_h(y)$ such that $g'_{(x,v_i)\rightarrow (x,y)}$ is a $(1,0)$-dominating function on $G\circ H$.
Therefore, $\gamma_{(1,0)}^s(G\circ H)\leq \omega(g')=\omega(g)=\gamma_{(2,2,0)}(G)$.

\vspace{.2cm}
\noindent
Case $5.$ $\gamma_{(1,0)}^s(H)>\gamma(H)=2$. Since $\gamma_{(1,0)}^s(H)\geq 3$, if $x\in X_0\cup X_1$, then  $f'(N(x))=f(N(V(H_x))\setminus V(H_x))\ge 2$, while  if $x\in X_2$, then $f'(N(x))=f(N(V(H_x))\setminus V(H_x))\ge 1$. Hence,  $f'$ is a  $(2,2,1)$-dominating function on $G$. Therefore, $\gamma_{(1,0)}^s(G\circ H)=\omega(f)=\omega(f')\ge \gamma_{(2,2,1)}(G)$.

In order to prove that $\gamma_{(1,0)}^s(G\circ H)\le \gamma_{(2,2,1)}(G)$, let $g(W_0,W_1,W_2)$ be a $\gamma_{(2,2,1)}(G)$-function.  
If $\gamma_{(1,0)}^s(H)>\gamma(H)=2$, then for any dominating set $S=\{v_1,v_2\}$ of $H$, the function $g'(W_0',W_1')$, defined by $W_1'=W_1\times \{v_1\}\cup W_2\times S$, is a $(2,1)$-dominating function on $G\circ H$. Hence, $\gamma_{(2,1)}(G\circ H)\leq \omega(g')=\omega(g)=\gamma_{(2,2,1)}(G)$, and so 
Corollary \ref{CorollarySecureSumaUno} leads to 
$\gamma_{(1,0)}^s(G\circ H)\leq \gamma_{(2,1)}(G\circ H)\leq \gamma_{(2,2,1)}(G).$

\vspace{.2cm}
\noindent
Case $6.$  $\gamma_{(1,0)}^s(H)=\gamma(H)=3$. As in Case 5, we deduce that  $\gamma_{(1,0)}^s(G\circ H)\ge \gamma_{(2,2,1)}(G)$.

In order to prove that $\gamma_{(1,0)}^s(G\circ H)\le \gamma_{(2,2,2)}(G)$, let $g(W_0,W_1,W_2)$ be a $\gamma_{(2,2,2)}(G)$-function and $S=\{v_1,v_2\}\subseteq V(H)$. 
The function $g'(W_0',W_1')$, defined by $W_1'=W_1\times \{v_1\}\cup W_2\times S$, is a $(2,2)$-dominating function on $G\circ H$. Hence, $\gamma_{(2,2)}(G\circ H)\leq \omega(g')=\omega(g)=\gamma_{(2,2,2)}(G)$. Therefore, 
Corollary \ref{CorollarySecureSumaUno} leads to 
$\gamma_{(1,0)}^s(G\circ H)\leq \gamma_{(1,1)}^s(G\circ H)\le \gamma_{(2,2)}(G\circ H)\leq \gamma_{(2,2,2)}(G).$


\vspace{.2cm}
\noindent
Case $7.$ $\gamma_{(1,0)}^s(H)\geq 4$ and $\gamma(H)\geq 3$. In this case, it is easy to check that $f'(N(x))=f(N(V(H_x))\setminus V(H_x))\ge 2$ for every $x\in V(G)$. Therefore,  $f'$ is a $(2,2,2)$-dominating function on $G$, which implies that  $\gamma_{(1,0)}^s(G\circ H)=\omega(f)=\omega(f')\ge \gamma_{(2,2,2)}(G)$.

Finally, as in Case 6, we can deduce that $\gamma_{(2,2)}(G\circ H)\leq  \gamma_{(2,2,2)}(G)$, and so $\gamma_{(1,0)}^s(G\circ H)\leq \gamma_{(1,1)}^s(G\circ H)\le \gamma_{(2,2)}(G\circ H)\leq \gamma_{(2,2,2)}(G).$
\end{proof}

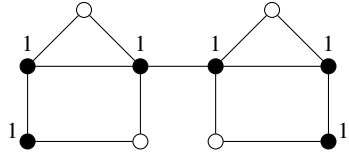
\begin{figure}[ht] 
\centering
\begin{tikzpicture}[scale=.5, transform shape]

\node [draw, shape=circle, fill=black] (a1) at  (0,0) {};
\node at (0,0.6) {\Large $1$};

\node [draw, shape=circle] (a2) at  (-1.5,1.5) {};

\node [draw, shape=circle, fill=black] (a3) at  (-3,0) {};
\node at (-3,0.6) {\Large $1$};

\node [draw, shape=circle, fill=black] (a4) at  (-3,-2) {};
\node at (-3.4,-1.7) {\Large $1$};
\node [draw, shape=circle] (a5) at  (0,-2) {};

\node [draw, shape=circle, fill=black] (b1) at  (2,0) {};
\node at (2,0.6) {\Large $1$};
\node [draw, shape=circle] (b2) at  (3.5,1.5) {};

\node [draw, shape=circle, fill=black] (b3) at  (5,0) {};
\node at (5,0.6) {\Large $1$};
\node [draw, shape=circle, fill=black] (b4) at  (5,-2) {};
\node at (5.4,-1.7) {\Large $1$};

\node [draw, shape=circle] (b5) at  (2,-2) {};

\draw(a1)--(a2)--(a3)--(a4)--(a5)--(a1);
\draw(b1)--(b2)--(b3)--(b4)--(b5)--(b1);
\draw(b3)--(b1)--(a1)--(a3);

\end{tikzpicture}
\caption{A graph $G$, where the labels asigned to the vertices  correspond to the positive weights assigned by a $\gamma_{(1,1,0)}^s(G)$-function.}\label{fig-new}
\end{figure}

\begin{figure}[ht]
\centering
\begin{tikzpicture}[scale=1.7]
\def\edgeLength{.7} 

\node[draw, shape=circle, fill=black, scale=.5] at (0,0cm)(v1){};
\node at ([shift={(0,.2)}]v1) {$2$};

\node[draw, shape=circle, fill=black, scale=.5] at (0:\edgeLength cm)(v2){};
\node at ([shift={(0,.2)}]v2) {$2$};
\node[draw, shape=circle, fill=white, scale=.5] at ([shift={(45:\edgeLength)}]v2) (v3){};
\node[draw, shape=circle, fill=white, scale=.5] at ([shift={(0:\edgeLength/2*sqrt(2))}]v2) (v4){};
\node[draw, shape=circle, fill=white, scale=.5] at ([shift={(-45:\edgeLength)}]v2) (v5){};

\node[draw, shape=circle, fill=black, scale=.5] at (180:\edgeLength cm)(v6){};
\node at ([shift={(0,.2)}]v6) {$2$};
\node[draw, shape=circle, fill=white, scale=.5] at ([shift={(135:\edgeLength)}]v6) (v7){};
\node[draw, shape=circle, fill=white, scale=.5] at ([shift={(180:\edgeLength/2*sqrt(2))}]v6) (v8){};
\node[draw, shape=circle, fill=white, scale=.5] at ([shift={(-135:\edgeLength)}]v6) (v9){};

\foreach \ind in {1,3,4,5}
{
\draw (v2)--(v\ind);
}

\foreach \ind in {1,7,8,9}
{
\draw (v6)--(v\ind);
}

\node at ([shift={(0,-.5)}]v1) {$G_1$};

\end{tikzpicture}
\hspace{.7cm}
\begin{tikzpicture}[scale=1.7]
\def\edgeLength{.7} 

\node[draw, shape=circle, fill=black, scale=.5] at (0,0cm)(v1){};
\node at ([shift={(0,.2)}]v1) {$1$};

\node[draw, shape=circle, fill=black, scale=.5] at (0:\edgeLength cm)(v2){};
\node at ([shift={(0,.2)}]v2) {$2$};
\node[draw, shape=circle, fill=white, scale=.5] at ([shift={(45:\edgeLength)}]v2) (v3){};
\node[draw, shape=circle, fill=white, scale=.5] at ([shift={(0:\edgeLength/2*sqrt(2))}]v2) (v4){};
\node[draw, shape=circle, fill=white, scale=.5] at ([shift={(-45:\edgeLength)}]v2) (v5){};

\node[draw, shape=circle, fill=black, scale=.5] at (180:\edgeLength cm)(v6){};
\node at ([shift={(0,.2)}]v6) {$1$};

\node[draw, shape=circle, fill=black, scale=.5] at (180:2*\edgeLength cm)(v7){};
\node at ([shift={(0,.2)}]v7) {$2$};
\node[draw, shape=circle, fill=white, scale=.5] at ([shift={(135:\edgeLength)}]v7) (v8){};
\node[draw, shape=circle, fill=white, scale=.5] at ([shift={(180:\edgeLength/2*sqrt(2))}]v7) (v9){};
\node[draw, shape=circle, fill=white, scale=.5] at ([shift={(-135:\edgeLength)}]v7) (v10){};

\foreach \ind in {1,3,4,5}
{
\draw (v2)--(v\ind);
}

\foreach \ind in {6,8,9,10}
{
\draw (v7)--(v\ind);
}

\draw (v6)--(v1);

\node at ($(v1)!0.5!(v6)$)(vlabel) {};
\node at ([shift={(0,-.5cm)}]vlabel) {$G_2$};

\end{tikzpicture}
\vspace{.7cm}
\begin{tikzpicture}[scale=1.7]
\def\edgeLength{.7} 

\node[draw, shape=circle, fill=black, scale=.5] at (0,0 cm)(v11){};
\node at ([shift={(0,.2)}]v11) {$2$};
\node[draw, shape=circle, fill=black, scale=.5] at ([shift={(0:\edgeLength)}]v11)(v12){};
\node at ([shift={(0,.2)}]v12) {$2$};
\node[draw, shape=circle, fill=white, scale=.5] at ([shift={(45:\edgeLength)}]v12) (v13){};
\node[draw, shape=circle, fill=white, scale=.5] at ([shift={(0:\edgeLength/2*sqrt(2))}]v12) (v14){};
\node[draw, shape=circle, fill=white, scale=.5] at ([shift={(-45:\edgeLength)}]v12) (v15){};

\node[draw, shape=circle, fill=black, scale=.5] at ([shift={(180:\edgeLength)}]v11)(v16){};
\node at ([shift={(0,.2)}]v16) {$2$};
\node[draw, shape=circle, fill=white, scale=.5] at ([shift={(135:\edgeLength)}]v16) (v17){};
\node[draw, shape=circle, fill=white, scale=.5] at ([shift={(180:\edgeLength/2*sqrt(2))}]v16) (v18){};
\node[draw, shape=circle, fill=white, scale=.5] at ([shift={(-135:\edgeLength)}]v16) (v19){};

\foreach \ind in {11,13,14,15}
{
\draw (v12)--(v\ind);
}

\foreach \ind in {11,17,18,19}
{
\draw (v16)--(v\ind);
}

\node[draw, shape=circle, fill=black, scale=.5] at (5.3*\edgeLength,0cm)(v1){};
\node at ([shift={(0,.2)}]v1) {$1$};
\node[draw, shape=circle, fill=black, scale=.5] at ([shift={(0:\edgeLength)}]v1)(v2){};
\node at ([shift={(0,.2)}]v2) {$2$};
\node[draw, shape=circle, fill=white, scale=.5] at ([shift={(45:\edgeLength)}]v2) (v3){};
\node[draw, shape=circle, fill=white, scale=.5] at ([shift={(0:\edgeLength/2*sqrt(2))}]v2) (v4){};
\node[draw, shape=circle, fill=white, scale=.5] at ([shift={(-45:\edgeLength)}]v2) (v5){};

\node[draw, shape=circle, fill=black, scale=.5] at ([shift={(180:\edgeLength)}]v1)(v6){};
\node at ([shift={(0,.2)}]v6) {$1$};

\node[draw, shape=circle, fill=black, scale=.5] at ([shift={(180:2*\edgeLength)}]v1)(v7){};
\node at ([shift={(0,.2)}]v7) {$2$};
\node[draw, shape=circle, fill=white, scale=.5] at ([shift={(135:\edgeLength)}]v7) (v8){};
\node[draw, shape=circle, fill=white, scale=.5] at ([shift={(180:\edgeLength/2*sqrt(2))}]v7) (v9){};
\node[draw, shape=circle, fill=white, scale=.5] at ([shift={(-135:\edgeLength)}]v7) (v10){};

\foreach \ind in {1,3,4,5}
{
\draw (v2)--(v\ind);
}

\foreach \ind in {6,8,9,10}
{
\draw (v7)--(v\ind);
}

\draw (v6)--(v1);

\draw (v9)--(v14);

\node at ($(v15)!0.5!(v10)$)(vlabel) {};
\node at ([shift={(0,-.3cm)}]vlabel) {$G_3$};
\node at ([shift={(0,0.3cm)}]v13) {};
\end{tikzpicture}

\caption{For any $i\in \{1,2,3\}$, the labels asigned to the vertices of $G_i$ correspond to the positive weights assigned by a $\gamma_{(1,0)}^s(G_i\circ H)$-function to the different copies of $H$ in $G_i\circ H$, where $H$ is any graph with $\gamma_{(1,0)}^s(H)=\gamma(H)=3$.}\label{LEXIC-raro-Secure}
\end{figure}
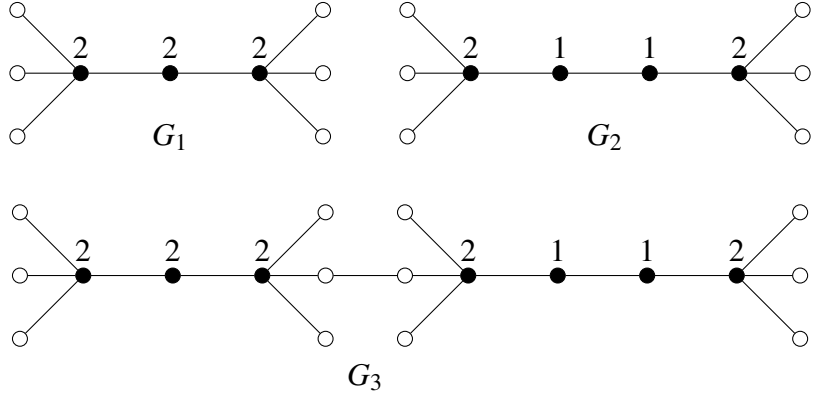

In order to show the behaviour of $\gamma_{(1,0)}^s(G\circ H)$ when $\gamma_{(1,0)}^s(H)=\gamma(H)=2$, we consider the following examples. 
For the graph $G$ shown in Figure \ref{fig-new}, $\gamma_{(1,0)}^s(G\circ H)=\gamma_{(1,1,0)}^s(G)=6<8=\gamma_{(2,2,0)}(G)$, while  
$\gamma_{(1,1,0)}^s(G\cup C_4)=9<10=\gamma_{(1,0)}^s((G\cup C_4)\circ H)<12=\gamma_{(2,2,0)}(G\cup C_4)$
and
$\gamma_{(1,0)}^s(C_4\circ H)=\gamma_{(2,2,0)}(C_4)=4 >3=\gamma_{(1,1,0)}^s(C_4)$.

Analogously, for the case $\gamma_{(1,0)}^s(H)=\gamma(H)=3$, we consider the graphs $G_1$, $G_2$ and $G_3$ illustrated in Figure \ref{LEXIC-raro-Secure}. 
 The weights shown in Figure \ref{LEXIC-raro-Secure} correspond to the weights assigned by a $\gamma_{(1,0)}^s(G_i\circ H)$-function to the different copies of $H$ in $G_i\circ H$. In particular, $\gamma_{(1,0)}^s(G_1\circ H)=\gamma_{(2,2,2)}(G_1)=6$, $\gamma_{(1,0)}^s(G_2\circ H)=\gamma_{(2,2,1)}(G_2)=6$ and $\gamma_{(2,2,1)}(G_3)=11<12=\gamma_{(1,0)}^s(G_3\circ H)<14=\gamma_{(2,2,2)}(G_3).$

We now discuss some particular cases of Theorem \ref{teo-equality-gamma(1,0)}.

\begin{corollary}
The following statements hold for any  integers $n,r\ge 4$ and a nontrivial graph $H$.

\begin{itemize}
\item $\gamma_{(1,0)}^s(K_n\circ H)=\left\{ \begin{array}{ll}
             1 & \text{if } H \text{ is a complete graph},\\[5pt]
             2 & \text{if } \,  \gamma_{(1,0)}^s(H)>\gamma(H)=1 \, \text{ or } \, \gamma_{(1,0)}^s(H)=\gamma(H)=2 ,\\[5pt]
             3 & \mbox{otherwise.}
                                  \end{array}\right.$
                                  \\
     \\                               
\item $\gamma_{(1,0)}^s(K_{1,n-1}\circ H)=\left\{ \begin{array}{ll}
             2 & \text{if } \,  \gamma_{(1,0)}^s(H)\leq 2,\\[5pt]
             4 & \text{if } \,  \gamma_{(1,0)}^s(H)\geq 4 \, \text{ and } \, \gamma(H)\geq 3 ,\\[5pt]
             3 & \mbox{otherwise.}
                                  \end{array}\right.$
                   \\  
                   \\             
                                  \item $\gamma_{(1,0)}^s(K_{2,n}\circ H)=\left\{ \begin{array}{ll}
             2 & \text{if } H \text{ is a complete graph},\\[5pt]
             4 & \text{if } \,  \gamma_{(1,0)}^s(H)\geq 2 \, \text{ and } \, \gamma(H)\geq 2 ,\\[5pt]
             3 & \mbox{otherwise.}
                                  \end{array}\right.$
\\
\\
\item $\gamma_{(1,0)}^s(K_{3,n}\circ H)=\left\{ \begin{array}{ll}
             3 & \text{if } H \text{ is a complete graph},\\[5pt]
                       4 & \mbox{otherwise.}
                                  \end{array}\right.$
\\
\\
\item $\gamma_{(1,0)}^s(K_{n,r}\circ H)=4$.
\end{itemize}
\end{corollary}

\begin{proof}
It is not difficult to see that if $\gamma_{(1,0)}^s(H)=\gamma(H)=2$, then $\gamma_{(1,0)}^s(K_{2,n}\circ H)=\gamma_{(2,2,0)}(K_{2,n})=4$.
For the remaining cases, the result follows from Theorem \ref{teo-equality-gamma(1,0)} by considering the following facts.

\begin{itemize}
\item $\gamma_{(1,0,0)}^s(K_n)=1$, $\gamma_{(2,1,0)}(K_n)=\gamma_{(2,1,1)}(K_n)=\gamma_{(1,1,0)}^s(K_n)=\gamma_{(2,2,0)}(K_n)=2$ and $\gamma_{(2,2,1)}(K_n)=\gamma_{(2,2,2)}(K_n)=3$. 
\\
\item  $\gamma_{(1,0,0)}^s(K_{1,n-1})=\gamma_{(2,1,0)}(K_{1,n-1})=\gamma_{(1,1,0)}^s(K_{1,n-1})=\gamma_{(2,2,0)}(K_{1,n-1})=2$, $\gamma_{(2,1,1)}(K_{1,n-1})=\gamma_{(2,2,1)}(K_{1,n-1})=3$ and $\gamma_{(2,2,2)}(K_{1,n-1})=4$. 
\\
\item $\gamma_{(1,0,0)}^s(K_{2,n})=2$, $\gamma_{(2,1,0)}(K_{2,n})=\gamma_{(2,1,1)}(K_{2,n})=3$ and $\gamma_{(2,2,1)}(K_{2,n})=\gamma_{(2,2,2)}(K_{2,n})=4$.
\\
\item $\gamma_{(1,0,0)}^s(K_{3,n})=3$ and $\gamma_{(2,1,0)}(K_{3,n})=\gamma_{(2,1,1)}(K_{3,n})=\gamma_{(1,1,0)}^s(K_{3,n})=\gamma_{(2,2,0)}(K_{3,n})=\gamma_{(2,2,1)}(K_{3,n})=\gamma_{(2,2,2)}(K_{3,n})=4$.
\\
\item $\gamma_{(1,0,0)}^s(K_{n,r})=\gamma_{(2,1,0)}(K_{n,r})=\gamma_{(1,1,0)}^s(K_{n,r})=\gamma_{(2,2,0)}(K_{n,r})=\gamma_{(2,1,1)}(K_{n,r})=\gamma_{(2,2,1)}(K_{n,r})=\gamma_{(2,2,2)}(K_{n,r})=4$.
\end{itemize}
\end{proof}

Next we consider the particular case when $G$ is a path. As we will see, this case  has been partially studied in previous works. 

\newpage

\begin{theorem}\label{teo-equality-gamma(1,0)-Paths}
For any integer $n\ge 6$ and any nontrivial graph $H$, the following statements hold.
 
\begin{enumerate}
\item[{\rm (i)}]\mbox{\rm \cite{MR1991720}} If $\gamma_{(1,0)}^s(H)=1$, i.e., $H$ is a complete graph, then $\gamma_{(1,0)}^s(P_n\circ H)=\gamma_{(1,0,0)}^s(P_n)=
\left\lceil  \frac{3n}{7} \right\rceil$.
\\
\item[{\rm (ii)}]\mbox{\rm \cite{SecureLexicographicDMGT}} If $\gamma_{(1,0)}^s(H)=2$ and $\gamma(H)=1$, then $\gamma_{(1,0)}^s(P_n\circ H)=
2\left\lceil  \frac{n}{3} \right\rceil$.
\\
\item[{\rm (iii)}]\mbox{\rm \cite{SecureLexicographicDMGT}} If $\gamma_{(1,0)}^s(H)\ge 3$ and $\gamma(H)=1$, then $\gamma_{(1,0)}^s(P_n\circ H)= \left\{ \begin{array}{ll}
             \frac{2n}{3}+1 & \text{if } \,  n\equiv 0 \pmod 3,\\[5pt]
             2\lceil\frac{n}{3}\rceil & \mbox{otherwise.}
                                  \end{array}\right.$ 
                                  \\
\item[{\rm (iv)}]\mbox{\rm \cite{SecureLexicographicDMGT}} If $\gamma_{(1,0)}^s(H)=\gamma(H)=2$, then $\gamma_{(1,0)}^s(P_n\circ H)=2\left\lfloor \frac{n+2}{3} \right\rfloor$.
\\
\item[{\rm (v)}] If $\gamma_{(1,0)}^s(H)>\gamma(H)=2$, then $$\gamma_{(1,0)}^s(P_n\circ H)=\gamma_{(2,2,1)}(P_n)=\left\{ \begin{array}{ll}
             n-\lfloor\frac{n}{7}\rfloor+1 & \text{if } \,  n\equiv 1,2 \pmod 7,\\[5pt]
             n-\lfloor\frac{n}{7}\rfloor & \mbox{otherwise.}
                                  \end{array}\right.$$
\item[{\rm (vi)}] If  $\gamma_{(1,0)}^s(H)=\gamma(H)=3$, then $$\gamma_{(1,0)}^s(P_n\circ H)= \left\{ \begin{array}{ll}
             n-\left\lfloor \frac{n}{11}\right \rfloor +1 & \text{if } \,  n\equiv 1,2,5 \pmod {11},\\[5pt]
             n-\left\lfloor \frac{n}{11}\right \rfloor  & \mbox{otherwise.}
                                  \end{array}\right.$$
\item[{\rm (vii)}] If $\gamma_{(1,0)}^s(H)\geq 4$ and $\gamma(H)\geq 3$, then $$\gamma_{(1,0)}^s(P_n\circ H)=\gamma_{(2,2,2)}(P_n)=\left\{ \begin{array}{ll}
n & if \, n\equiv 0\pmod 4,\\[5pt]
n+1 & if \,  n\equiv 1,3\pmod 4,\\[5pt]
n+2  & if \,  n\equiv 2\pmod 4.
\end{array}\right. $$
\end{enumerate}
\end{theorem}

\begin{proof}

The proofs of (v) and (vii) are derived by combining  Theorem \ref{teo-equality-gamma(1,0)} with the values of $\gamma_{(2,2,1)}(P_n)$ and $\gamma_{(2,2,2)}(P_n)$ obtained in \cite{w-domination}. It remains to prove (vi).

Assume that $\gamma_{(1,0)}^s(H)=\gamma(H)=3$ and let $f(V_0,V_1)$ be a $\gamma_{(1,0)}^s(P_n\circ H)$-function which satisfies Lemma~\ref{lem-vertice<=2-secure}. Let  $f'(X_0,X_1,X_2)$ be the function defined on $P_n$ by $X_1=\{x\in V(P_n): f(V(H_x))=1\}$ and $X_2=\{x\in V(P_n): f(V(H_x))=2\}$. Notice that $\gamma_{(1,0)}^s(P_n\circ H)=\omega(f)=\omega(f')$.
Let $n=11q+r$ with $r\in \{0,\dots, 10\}$. With this notation in mind, we proceed by induction on $q$.

It is not difficult to check that the result follows for $q=0$ and $r \in \{6,\dots, 10\}$. For these cases, possible sequences of weights assigned by $f'$ to consecutive vertices of $P_r$ are  $021120$, $1200220$, $02200220$, $021012120$ and $0210121012$, respectively.
The result also follows for $q=1$ and $r=0$, i.e., $n=11$. In this case, the only possible sequence  of weights assigned by $f'$ to consecutive vertices of $P_{11}$ is $02101210120$. The certainty that the result holds for $q=1$ and $r \in \{3,4,6,\dots, 10\}$ comes from a computer search. For these cases, since $P_{11+r}$ can be obtained by connecting a leaf of $P_{11}$ with a leaf $P_r$, possible sequences of weights assigned by $f'$ to consecutive vertices of $P_{11+r}$ are obtained by concatenating the sequences of weights associated to $P_{11}$ and $P_r$, where the sequences for $r=3$ and $r=4$ are $120$, $0220$, respectively. In summary, for any 
$r \in \{0,3,4,6,\dots, 10\}$, we have that 
$\gamma_{(1,0)}^s(P_{11+r}\circ H)=10+r$. Analogously, among the possible sequences of weights assigned by $f'$ to consecutive vertices of $P_{12}$, $P_{13}$ and $P_{16}$,
a computer search gives, for instance,  $022101200220$, $0210210220120$ and $0211200220021120$, respectively. Thus, 
$\gamma_{(1,0)}^s(P_{12}\circ H)=12$, $\gamma_{(1,0)}^s(P_{13}\circ H)=13$ and $\gamma_{(1,0)}^s(P_{16}\circ H)=16$, which completes the base case.

Assume that $q\ge 2$ and the statement holds for any $q'$ such that
$1\le q'\le q$. We differentiate two cases.

\vspace{.2cm}
\noindent
Case $1.$  $r\not\in \{1,2,5\}$. In this case,   $$\gamma_{(1,0)}^s(P_{11q+r}\circ H)=10q+r.$$
Let $n=11(q+1)+r$, $k_1=\gamma_{(1,0)}^s(P_{n}\circ H)$, $P_{n}=x_1x_2\dots x_{n}$ and $k_2=f'(\{x_1,\dots, x_{11}\})$. 
Suppose that $k_1<10(q+1)+r$.
Since $02101210120$ is the only possible sequence of weights assigned by $f'$ to consecutive vertices of $P_{11}$, we have that $k_2\ge 10$, and 
$$
(k_2-10)+f'(\{x_{12},\dots , x_n\})=k_1-10<10q+r.
$$
Since the function $g$, defined as  $g(V(H_{x_{13}}))=f(V(H_{x_{13}}))+k_2-10$ and 
$g(V(H_{x_{i}}))=f(V(H_{x_{i}}))$ for every $i\in \{12\}\cup \{14,\dots, n\}$,  is a secure $(1,0)$-dominating function on the subgraph of  $P_{n}\circ H$ induced by $\{x_{12},\dots , x_n\}\times V(H)$, we can conclude that $\gamma_{(1,0)}^s(P_{11q+r}\circ H)\le \omega(g)=k_1-10<10q+r$, which contradicts the hypothesis. Thus,  
$\gamma_{(1,0)}^s(P_{n}\circ H)\ge  10(q+1)+r$.
To conclude the proof, we only need to observe that 
$P_{n}$ is obtained by connecting a leaf of $P_{11q+r}$ with a leaf of $P_{11}$, and so, by hypothesis, 
$\gamma_{(1,0)}^s(P_{n}\circ H)\le \gamma_{(1,0)}^s(P_{11q+r}\circ H)+\gamma_{(1,0)}^s(P_{11}\circ H)=10(q+1)+r$.
Therefore, the proof of this case is complete.

\vspace{.2cm}
\noindent
Case $2.$  $r\in  \{1,2,5\}$. This case is completely analogous to Case 1. The only difference is the induction hypothesis, which states that  $\gamma_{(1,0)}^s(P_{11q+r}\circ H)=10q+r+1.$  
Hence, we take $n=11(q+1)+r$ and following the procedure described above,  we deduce that   $\gamma_{(1,0)}^s(P_{n}\circ H)=10(q+1)+r+1$, which completes the proof.
\end{proof}

The next result concerns the case when $G$ is a cycle.

\begin{theorem}\label{teo-equality-gamma(1,0)-Cycles}
For any integer $n\ge 6$ and a nontrivial graph $H$, the following statements hold.
 
\begin{enumerate}
\item[{\rm (i)}]\mbox{\rm \cite{MR1991720}} If $\gamma_{(1,0)}^s(H)=1$, i.e., $H$ is a complete graph, then $\gamma_{(1,0)}^s(C_n\circ H)=\gamma_{(1,0,0)}^s(C_n)=
\left\lceil  \frac{3n}{7} \right\rceil$.
\\
\item[{\rm (ii)}]\mbox{\rm \cite{SecureLexicographicDMGT}} If $\gamma_{(1,0)}^s(H)\ge 2$ and $\gamma(H)=1$, then $\gamma_{(1,0)}^s(C_n\circ H)=  \left\lceil\frac{2n}{3}\right\rceil $.
\\
\item[{\rm (iii)}]\mbox{\rm \cite{SecureLexicographicDMGT}} If $\gamma_{(1,0)}^s(H)=\gamma(H)=2$, then $\gamma_{(1,0)}^s(C_n\circ H)=2\left\lfloor \frac{n+2}{3} \right\rfloor$.
\\
\item[{\rm (iv)}] If $\gamma_{(1,0)}^s(H)>\gamma(H)=2$, then $$\gamma_{(1,0)}^s(C_n\circ H)=\gamma_{(2,2,1)}(C_n)=\left\{ \begin{array}{ll}
             n-\lfloor\frac{n}{7}\rfloor+1 & \text{if } \,  n\equiv 1,2 \pmod 7,\\[5pt]
             n-\lfloor\frac{n}{7}\rfloor & \mbox{otherwise.}
                                  \end{array}\right.$$
                         
\item[{\rm (v)}] If  $\gamma_{(1,0)}^s(H)=\gamma(H)=3$, then $$\gamma_{(1,0)}^s(C_n\circ H)= \left\{ \begin{array}{ll}
             n-\left\lfloor \frac{n}{11}\right \rfloor +1 & \text{if } \,  n\equiv 1,2,5 \pmod {11},\\[5pt]
             n-\left\lfloor \frac{n}{11}\right \rfloor  & \mbox{otherwise.}
                                  \end{array}\right.$$
                                  
\item[{\rm (vi)}] If $\gamma_{(1,0)}^s(H)\geq 4$ and $\gamma(H)\geq 3$, then $\gamma_{(1,0)}^s(C_n\circ H)=\gamma_{(2,2,2)}(C_n)=n. $
\end{enumerate}
\end{theorem}
 
\begin{proof}

The proofs of (iv) and (vi) are derived by combining  Theorem \ref{teo-equality-gamma(1,0)} with the values of $\gamma_{(2,2,1)}(C_n)$ and $\gamma_{(2,2,2)}(C_n)$ obtained in \cite{w-domination}. It remains to prove (v).

Assume that $\gamma_{(1,0)}^s(H)=\gamma(H)=3$ and let $f(V_0,V_1)$ be a $\gamma_{(1,0)}^s(C_n\circ H)$-function which satisfies Lemma~\ref{lem-vertice<=2-secure}. Let  $f'(X_0,X_1,X_2)$ be the function defined on $C_n$ by $X_1=\{x\in V(C_n): f(V(H_x))=1\}$ and $X_2=\{x\in V(C_n): f(V(H_x))=2\}$. Notice that $\gamma_{(1,0)}^s(C_n\circ H)=\omega(f)=\omega(f')$.

Let $V(C_n)=\{x_0,\dots, x_{n-1}\}$, where consecutive vertices are adjacent and the addition of subscripts is taken modulo $n$.
If there exists $x_i\in V(C_n)$ such that $f(x_i)=f(x_{i+1})=0$, then $\gamma_{(1,0)}^s(C_n\circ H)=\gamma_{(1,0)}^s(P_n\circ H)$ and we derive the result by Theorem \ref{teo-equality-gamma(1,0)-Paths}. From now on, we assume that for any  $x_i\in X_0$ we have that $f'(x_{i-1})>0$ and $f'(x_{i+1})>0$, which implies that $f'(x_{i-2})+f(x_{i-1})\ge 3$ and $f'(x_{i+1})+f(x_{i+2})\ge 3$. Now, for any $x_i\in X_0$ we define $S_i=\{x_{i},x_{i+1},x_{i+2}\}$ and $S=V(C_n)\setminus \left(\cup_{x_i\in X_0}S_i\right)$. Notice that $f(S_i)\ge 3=|S_i|$ for every $x_i\in X_0$ and $f(x_j)>0$ for every $x_j\in S$. 
Hence, by Proposition \ref{obs-subgraph-general-Secure}, $$\gamma_{(1,0)}^s(P_n\circ H)\ge \gamma_{(1,0)}^s(C_n\circ H) =\omega(f')=\sum_{x_i\in X_0}f'(S_i)+f'(S)\ge n.$$
This implies that $n\in \{6,\dots, 10, 12, 13, 15\}$ and $\gamma_{(1,0)}^s(C_n\circ H)=n$. Therefore, the result follows.
\end{proof}

As a direct consequence of Theorems \ref{teo-equality-gamma(1,0)}, \ref{teo-equality-gamma(1,0)-Paths} and  \ref{teo-equality-gamma(1,0)-Cycles}  we derive the following result.

\begin{proposition}\label{Prop(2,1,1)-(1,1,0)s-paths}
The following statements hold for any integer $n\ge 4$.

\begin{itemize}
\item $\gamma_{(2,1,1)}(P_n)=\left\{ \begin{array}{ll}
             \frac{2n}{3}+1 & \text{if } \,  n\equiv 0 \pmod 3,\\[5pt]
             2\lceil\frac{n}{3}\rceil & \mbox{otherwise.}
                                  \end{array}\right.$
        \\
        \\                            
 \item $\gamma_{(2,1,1)}(C_n)=\lceil\frac{2n}{3}\rceil.$
\end{itemize}
\end{proposition}

\subsection{Weak Roman domination}

This subsection is devoted to study the weak Roman domination of lexicographic product graphs.  To this end, we need the following tools.

\begin{remark}{\rm \cite{Valveny2017}}\label{remark-gamma_r=1-non_K_n}
Given a noncomplete graph $G$, the following statements are equivalent.
\begin{itemize}
\item $\gamma_{(1,0,0)}^s(G)=2.$
\\
\item $\gamma(G)=1$ or $\gamma_{(1,0)}^s(G)=2.$
\end{itemize}
\end{remark}

\begin{lemma}\label{lem-vertice<=2-weak}
For any  graph $G$ with no isolated vertex and any nontrivial graph $H$ with $\gamma(H)=1$, there exists a $\gamma_{(1,0,0)}^s(G\circ H)$-function $f$ such that $f(V(H_u))\leq 2$, for every $u\in V(G)$.
\end{lemma}

\begin{proof}
Given a secure $(1,0,0)$-dominating function $f$ on $G\circ H$, we define $R_f=\{x\in V(G): \, f(V(H_x))\ge 3\}$. Let $f(V_0,V_1,V_2)$ be a $\gamma_{(1,0,0)}^s(G\circ H)$-function such that $|R_f|$ is minimum among all $\gamma_{(1,0,0)}^s(G\circ H)$-functions.

Suppose that $|R_f|\geq 1$.  Let $u\in R_f$, $u'\in N(u)$ and $v$ be a universal vertex of $H$. Notice that the function $f'$ on $G\circ H$, defined by $f'(u,v)=f'(V(H_{u}))=2$, $f'(V(H_{u'}))=\min\{2,f(V(H_{u'}))+1\}$, and $f'(x,y)=f(x,y)$ for every $x\in V(G)\setminus \{u,u'\}$ and $y\in V(H)$, is a secure $(1,0,0)$-dominating function on $G\circ H$ with $\omega(f')\le \omega(f)$ and $|R_{f'}|<|R_f|$, which is a contradiction. 
Therefore, $R_f=\varnothing $, and so the result follows. 
\end{proof}

Our goal is to show how the weak Roman domination number of $G\circ H$ is related to   $\gamma_{w}^s(G)$ or $\gamma_{w}(G)$ for certain vectors $w$ of three components. By  Theorems  \ref{teo-s=r-lexicographic} and \ref{teo-equality-gamma(1,0)}, 
 the outstanding case is when  $\gamma_{(1,0)}^s(H)\ge 3$ and $\gamma_{(1,0,0)}^s(H)=2$, which is equivalent to $\gamma(H)=1$ and $\gamma_{(1,0)}^s(H)\ge 3$, by Remark \ref{remark-gamma_r=1-non_K_n}. 
In fact, we will present the result with the only assumption on $H$ of being a noncomplete graph with $\gamma(H)=1$. 

\begin{theorem}\label{Th-WRD-Lex-gammaH=1}
For any graph $G$ with no isolated vertex and any noncomplete graph $H$ with $\gamma(H)=1$, 
 $$\gamma_{(1,0,0)}^s(G\circ H)=\gamma_{(2,1,0)}(G).$$
\end{theorem}

\begin{proof}
Let $H$ be a noncomplete graph with $\gamma(H)=1$ and  $f(V_0,V_1,V_2)$ a $\gamma_{(1,0,0)}^s(G\circ H)$-function which satisfies Lemma \ref{lem-vertice<=2-weak}.  Let  $f'(X_0,X_1,X_2)$ be the function defined on $G$ by $X_1=\{x\in V(G): f(V(H_x))=1\}$ and $X_2=\{x\in V(G): f(V(H_x))=2\}$.
Since $\gamma_{(1,0,0)}^s(H)=2$, if $x\in X_0$, then  $f'(N(x))=f(N(V(H_x))\setminus V(H_x))\ge 2$ and if $x\in X_1$, then  $f'(N(x))=f(N(V(H_x))\setminus V(H_x))\ge 1$. Hence,  $f'$ is a  $(2,1,0)$-dominating function on $G$, which implies that $\gamma_{(1,0,0)}^s(G\circ H)=\omega(f)=\omega(f')\ge \gamma_{(2,1,0)}(G)$.

Now, we show that $\gamma_{(1,0,0)}^s(G\circ H)\le\gamma_{(2,1,0)}(G)$. Let $v$ be a universal vertex of $H$. For any $\gamma_{(2,1,0)}(G)$-function $g(W_0,W_1,W_2)$, the function $g'(W_0',W_1',W_2')$, defined by $W_1'=W_1\times\{v\}$ and $W_2'=W_2\times\{v\}$, is a $(1,0,0)$-dominating function on $G\circ H$. Now, for any $(x,y)\in W_0'\setminus W_0\times V(H)$, we can see that  $g'_{(x,v)\rightarrow (x,y)}$ is a $(1,0,0)$-dominating function on $G\circ H$. Furthermore, for every $x\in W_0$ there exists $x'\in W_2\cap N(x)$ or two vertices $x',x''\in W_1\cap N(x)$, and so  $g'_{(x',v)\rightarrow (x,y)}$ is a $(1,0,0)$-dominating function on $G\circ H$ for every $(x,y)\in W_0\times V(H)$. Therefore, $g'$ is a secure $(1,0,0)$-dominating function on $G\circ H$, and as a consequence, $\gamma_{(1,0,0)}^s(G\circ H)\le\omega(g')=\omega(g)=\gamma_{(2,1,0)}(G)$. 
\end{proof}

We next present some particular cases of Theorem \ref{Th-WRD-Lex-gammaH=1}.

\begin{corollary}
Let  $n$ and $r$ be two  integers  such that $n\geq r\geq 2$. If $H$ is a noncomplete graph with 
$\gamma(H)=1$, then the following statements hold.
\begin{itemize}
\item $\gamma_{(1,0,0)}^s(K_n\circ H)=2$.
         \\                           
\item $\gamma_{(1,0,0)}^s(K_{1,n-1}\circ H)=2$.
\\
\item $\gamma_{(1,0,0)}^s(K_{n,r}\circ H)=\left\{ \begin{array}{ll}
             3 & \text{if }r=2,\\[5pt]
             4 & \mbox{otherwise.}
                                  \end{array}\right.$
\end{itemize}
\end{corollary}

\begin{proof}
 Theorem \ref{Th-WRD-Lex-gammaH=1} leads to  the result, as $\gamma_{(2,1,0)}(K_n)=\gamma_{(2,1,0)}(K_{1,n-1})=2$, $\gamma_{(2,1,0)}(K_{n,2})=3$ and  $\gamma_{(2,1,0)}(K_{n,r})=4$ whenever $r\ge 3$.
\end{proof}

To conclude this section, we would present the following result which shows that the study of the cases $G\cong P_n$ and $G\cong C_n$ is complete. 

\begin{theorem}{\rm \cite{SecureLexicographicDMGT}}
For any integer $n\ge 3$ and any noncomplete graph with $\gamma(H)=1$,
 $$\gamma_{(1,0,0)}^s(P_n\circ H)= 2 \left\lceil\frac{n}{3}\right\rceil  \quad \text{and} \quad  \gamma_{(1,0,0)}^s(C_n\circ H)=\left\lceil\frac{ 2n}{3}\right\rceil .$$ 
\end{theorem}

\subsection{Secure total domination and total weak Roman domination}
It was shown in \cite{TPlexicographic-2019} that for any graph $G$ with no isolated vertex and any nontrivial graph $H$ the secure total domination number of $G\circ H$ equals the total weak Roman domination number. 

\begin{theorem}\label{teo-st=tr-lexicographic}{\rm \cite{TPlexicographic-2019}} 
For any graph $G$ with no isolated vertex and any nontrivial graph $H$,
$$\gamma_{(1,1,1)}^s(G\circ H)=\gamma_{(1,1)}^s(G\circ H).$$
\end{theorem}

According to Theorem \ref{teo-st=tr-lexicographic}, we can restrict ourselves to study the 
secure total domination number of $G\circ H$. To this end,  we shall need the following lemma. 

\begin{lemma}\label{lem-st-lexicographic}{\rm \cite{TPlexicographic-2019}} 
For any graph $G$ with no isolated vertex and any nontrivial graph $H$, there exists a $\gamma_{(1,1)}^s(G\circ H)$-function $f$ satisfying that $f(V(H_u))\leq 2$  for every $u\in V(G)$.
\end{lemma} 

The following result shows that the secure total domination number of $G\circ H$ equals  $\gamma_{w}^s(G)$ or $\gamma_{w}(G)$ for certain vectors $w$ of three components. As we can expect, the decision on whether  the components of $w$ take specific values will depend on the value of $\gamma_{(1,1)}^s(H)$ and/or $\gamma(H)$. 

\begin{theorem}\label{SecureTotalLexicographic}
For a graph $G$ with no isolated vertex and a nontrivial graph $H$, the following statements hold.
 
\begin{enumerate}
\item[{\rm (i)}] If $\gamma_{(1,1)}^s(H)=2$, then $\gamma_{(1,1)}^s(G\circ H)=\gamma_{(1,1,0)}^s(G)$.
\\
\item[{\rm (ii)}] If $\gamma(H)=1$ and $\gamma_{(1,1)}^s(H)\ge 3$, then $\gamma_{(1,1)}^s(G\circ H)=\gamma_{(1,1,1)}^s(G)$.
\\
\item[{\rm (iii)}] If $\gamma(H)=2<\gamma_{(1,1)}^s(H)$, then $\gamma_{(1,1)}^s(G\circ H)=\gamma_{(2,2,1)}(G)$.
\\
\item[{\rm (iv)}] If $\gamma(H)\geq 3$, then $\gamma_{(1,1)}^s(G\circ H)=\gamma_{(2,2,2)}(G)$.
\end{enumerate}
\end{theorem}

\begin{proof}
Let $f(V_0,V_1)$ be a $\gamma_{(1,1)}^s(G\circ H)$-function which satisfies Lemma \ref{lem-st-lexicographic}.  Let $f'(X_0,X_1,X_2)$ be the function defined on $G$ by $X_1=\{x\in V(G): f(V(H_x))=1\}$ and $X_2=\{x\in V(G): f(V(H_x))=2\}$. With this notation in mind, we differentiate the following four cases.

\vspace{.2cm}
\noindent
Case $1.$ $\gamma_{(1,1)}^s(H)=2$. If $x\in X_0\cup X_1$, then  $f'(N(x))=f(N(V(H_x))\setminus V(H_x))\ge  1$, which implies that  $f'$ is a  $(1,1,0)$-dominating function on $G$. Now, for any  $(x,y)\in X_0\times  V(H)$,  there exists $(x',y')\in M_f(x,y)$ with $x'\in N(x)\cap(X_1\cup X_2)$. Hence, for any $u\in X_0\cup X_1\cup \{x'\}$ we have that $f'_{x'\rightarrow x}(N(u))=f_{(x',y')\rightarrow (x,y)}(N(V(H_u))\setminus V(H_u))\ge 1$, which implies that 
$f'$ is a secure $(1,1,0)$-dominating function on $G$. Therefore, $\gamma_{(1,1)}^s(G\circ H)=\omega(f)=\omega(f')\ge \gamma_{(1,1,0)}^s(G)$.

On the other side, 
for any $\gamma_{(1,1,0)}^s(G)$-function  $g(W_0,W_1,W_2)$ and any two universal vertices $v_1,v_2$ of $H$, the function $g'(W_0',W_1')$, defined by $W_1'=(W_1\times \{v_1\})\cup (W_2\times \{v_1,v_2\})$, is a $(1,1)$-dominating function on $G\circ H$. Now,  for any  $(x,y)\in W_0'\setminus W_0\times V(H)$, we can see that $g'_{(x,v_1)\rightarrow (x,y)}$ is a $(1,1)$-dominating function on $G\circ H$.
Furthermore, for every $x\in W_0$ there exists $x'\in M_{g}(x)$, and so $g'_{(x',v_1)\rightarrow (x,y)}$ is a $(1,1)$-dominating function on $G\circ H$
for every  $(x,y)\in W_0\times V(H)$.
Therefore, $\gamma_{(1,1)}^s(G\circ H)\leq \omega(g')=\omega(g)=\gamma_{(1,1,0)}^s(G)$.


\vspace{.2cm}
\noindent
Case $2.$ $\gamma(H)=1$ and $\gamma_{(1,1)}^s(H)\ge 3$. 
 Since $f(V(H_x)\le 2$ and $\gamma_{(1,1)}^s(H_x)\ge 3$ for any $x\in V(G)$, we have that $f'(N(x))=f(N(V(H_x))\setminus V(H_x))\ge 1$. 
 Thus, $f'$ is a  $(1,1,1)$-dominating function on $G$. 
 Now, for any  $(x,y)\in X_0\times  V(H)$,  there exists $(x',y')\in M_f(x,y)$ with $x'\in N(x)\cap(X_1\cup X_2)$. Hence, for any $u\in V(G)$ we have that $f'_{x'\rightarrow x}(N(u))=f_{(x',y')\rightarrow (x,y)}(N(V(H_u))\setminus V(H_u))\ge 1$, which implies that    
$f'$ is a secure $(1,1,1)$-dominating function on $G$. Therefore, $\gamma_{(1,1)}^s(G\circ H)=\omega(f)=\omega(f')\ge \gamma_{(1,1,1)}^s(G)$.

Now we show that 
$\gamma_{(1,1)}^s(G\circ H)\le \gamma_{(1,1,1)}^s(G)$.
Let $v$ be the universal vertex of $H$, which is unique, as $\gamma_{(1,1)}^s(H)\ge 3$.  For any $\gamma_{(1,1,1)}^s(G)$-function  $g(W_0,W_1,W_2)$,  the function $g'(W_0',W_1',W'_2)$, defined by $W_1'=W_1\times \{v\}$ and $W_2'=W_2\times \{v\}$, is a $(1,1,1)$-dominating function on $G\circ H$. Now,  for any  $(x,y)\in W_0'\setminus W_0\times V(H)$, we can see that $g'_{(x,v)\rightarrow (x,y)}$ is a $(1,1,1)$-dominating function on $G\circ H$.
Furthermore, for every $x\in W_0$ there exists $x'\in M_{g}(x)$, and so $g'_{(x',v)\rightarrow (x,y)}$ is a $(1,1,1)$-dominating function on $G\circ H$
for every  $(x,y)\in W_0\times V(H)$.
Therefore, $g'$ is a secure $(1,1,1)$-dominating function on $G\circ H$, and
by Theorem~\ref{teo-st=tr-lexicographic} we deduce that $\gamma_{(1,1)}^s(G\circ H)=\gamma_{(1,1,1)}^s(G\circ H)\leq \omega(g')=\omega(g)=\gamma_{(1,1,1)}^s(G)$.


\vspace{.2cm}
\noindent
Case $3.$ $\gamma(H)=2<\gamma_{(1,1)}^s(H)$. Assume first that $x\in X_0\cup X_1$.  Since $\gamma_{(1,1)}^s(H_x)\ge 3$, we have that $f'(N(x))=f(N(V(H_x))\setminus V(H_x))\ge 2$. Analogously, for any  $x\in X_2$ we deduce that  $f'(N(x))=f(N(V(H_x))\setminus V(H_x))\geq 1$. Therefore,  $f'$ is a  $(2,2,1)$-dominating function on $G$  and, as a consequence, $\gamma_{(1,1)}^s(G\circ H)=\omega(f)=\omega(f')\ge \gamma_{(2,2,1)}(G)$.

On the other side, for any $\gamma_{(2,2,1)}(G)$-function  $g(W_0,W_1,W_2)$ and any $\gamma(H)$-set  $\{v_1,v_2\}$, the function $g'(W_0',W_1')$, defined by $W_1'=(W_1\times \{v_1\})\cup (W_2\times \{v_1,v_2\})$, is a  $(1,1)$-dominating function on $G\circ H$. Now, for any $(x,y)\in W_0'\setminus  W_2\times V(H)$ and $x'\in N(x)\cap (W_1\cup W_2)$ we have that $g'_{(x',v_1)\rightarrow (x,y)}(N(u,v))\ge 1$ for every $(u,v)\in V(G\circ H)$. 
Moreover, for any $(x,y)\in W_0'\cap (W_2\times V(H))$  we have that $g'_{(x,v_1)\rightarrow (x,y)}(N(u,v))\ge 1$ or $g'_{(x,v_2)\rightarrow (x,y)}(N(u,v))\ge 1$ for every $(u,v)\in V(G\circ H)$. 
Therefore, $g'$ is a secure $(1,1)$-dominating function on $G\circ H$, which implies that  $\gamma_{(1,1)}^s(G\circ H)\le \omega(g')=\omega(g)=\gamma_{(2,2,1)}(G)$.
 
\vspace{0,2cm}
\noindent
Case $4.$ $\gamma(H)\ge 3$.
In this case, for every $x\in V(G)$, there exists  $y\in V(H)$ such that $f(N[(x,y)]\cap V(H_x))=0$. Hence,  $f'(N(x))=f(N(x,y)\setminus V(H_{x}))\ge 2$ for every $x\in V(G)$. Therefore,  $f'$ is a $(2,2,2)$-dominating function on $G$, and so  $\gamma_{(1,1)}^s(G\circ H)=\omega(f)=\omega(f')\ge \gamma_{(2,2,2)}(G)$.

Now, for any $\gamma_{(2,2,2)}(G)$-function  $g(W_0,W_1,W_2)$ and any $v\in V(H)$, the function $g'(W_0',W_1',W_2')$, defined by $W_2'=W_2\times \{v\}$ and $W_1'=W_1\times \{v\}$, is a secure $(2,2,2)$-dominating function on $G\circ H$. Therefore, $\gamma_{(2,2,2)}(G\circ H)\le \omega(g')$, and by Theorems \ref{PreliminaryBounds-w-SecureSegundTh} and \ref{teo-st=tr-lexicographic} we deduce that $\gamma_{(1,1)}^s(G\circ H)=\gamma_{(1,1,1)}^s(G\circ H)\le \gamma_{(2,2,2)}(G\circ H)\le \omega(g')=\omega(g)= \gamma_{(2,2,2)}(G)$.
According to the cases above, the result follows.
\end{proof}

The following result is a particular case of Theorem \ref{SecureTotalLexicographic}.

\begin{corollary}
The following statements hold for a nontrivial graph $H$ and any integers $n,r$ such that $n\geq r\geq 2$.

\begin{itemize}
\item $\gamma_{(1,1)}^s(K_n\circ H)=\left\{ \begin{array}{ll}
             2 & \text{if } \,  \gamma(H)=1,\\[5pt]
             3 & \mbox{otherwise.}
                                  \end{array}\right.$
\\ 
\\                                   
\item $\gamma_{(1,1)}^s(K_{1,n-1}\circ H)=\left\{ \begin{array}{ll}
             2 & \text{if } \,  \gamma_{(1,1)}^s(H)=2,\\[5pt]
             4 & \text{if } \,  \gamma(H)\geq 3,\\[5pt]
             
             3 & \mbox{otherwise.}
                                  \end{array}\right.$
\\
\\
\item $\gamma_{(1,1)}^s(K_{n,r}\circ H)=\left\{ \begin{array}{ll}
             3 & \text{if } \,  \gamma(H)=1 \text{ and } r=2,\\[5pt]
             4 & \mbox{otherwise.}
                                  \end{array}\right.$
\end{itemize}
\end{corollary}

\begin{proof}
The result follows from Theorem \ref{SecureTotalLexicographic} by considering the following facts.
\begin{itemize}
\item $\gamma_{(1,1,0)}^s(K_n)=\gamma_{(1,1,1)}^s(K_n)=2$ and  $\gamma_{(2,2,1)}(K_n)=\gamma_{(2,2,2)}(K_n)=3$. 
\\
\item  $\gamma_{(1,1,0)}^s(K_{1,n-1})=2$, $\gamma_{(1,1,1)}^s(K_{1,n-1})=\gamma_{(2,2,1)}(K_{1,n-1})=3$ and $\gamma_{(2,2,2)}(K_{1,n-1})=4$.
\\
\item If $r=2$, then $\gamma_{(1,1,0)}^s(K_{n,r})=\gamma_{(1,1,1)}^s(K_{n,r})=3$, while if $r\ge 3$, then $\gamma_{(1,1,0)}^s(K_{n,r})=\gamma_{(1,1,1)}^s(K_{n,r})=\gamma_{(2,2,1)}(K_{n,r})=\gamma_{(2,2,2)}(K_{n,r})=4$.
\end{itemize}
\end{proof}

By combining Theorem \ref{SecureTotalLexicographic} and some known results we derive the following results.

\begin{theorem}\label{SecureTotalLexicographic-Paths}
The following statements hold for any integer $n\ge 4$ and any nontrivial graph $H$.
\begin{enumerate}
\item[{\rm (i)}] If $\gamma_{(1,1)}^s(H)=2$, then $\gamma_{(1,1)}^s(P_n\circ H)=2\left\lceil \frac{n}{3}\right\rceil$.\\
\item[{\rm (ii)}] If $\gamma(H)=1$ and $\gamma_{(1,1)}^s(H)\ge 3$, then $\gamma_{(1,1)}^s(P_n\circ H)=\gamma_{(1,1,1)}^s(P_n)\stackrel{\mbox{\rm \cite{TWRDF2018}}}{=}\left\lceil \frac{5(n-2)}{7}\right\rceil +2.$ \\
\item[{\rm (iii)}] If $\gamma(H)=2<\gamma_{(1,1)}^s(H)$, then $\gamma_{(1,1)}^s(P_n\circ H)=\gamma_{(2,2,1)}(P_n)\stackrel{\mbox{\rm \cite{w-domination}}}{=} \left\{ \begin{array}{ll}
             n-\lfloor\frac{n}{7}\rfloor+1 & \text{if } \,  n\equiv 1,2 \pmod 7,\\[5pt]
             n-\lfloor\frac{n}{7}\rfloor & \mbox{otherwise.}
                                  \end{array}\right.$
\item[{\rm (iv)}] If $\gamma(H)\geq 3$, then $\gamma_{(1,1)}^s(P_n\circ H)=\gamma_{(2,2,2)}(P_n)\stackrel{\mbox{\rm \cite{w-domination}}}{=}\left\{ \begin{array}{ll}
n & if \, n\equiv 0\pmod 4,\\[5pt]
n+1 & if \,  n\equiv 1,3\pmod 4,\\[5pt]
n+2  & if \,  n\equiv 2\pmod 4.
\end{array}\right. $
\end{enumerate}
\end{theorem}

\begin{proof}
As indicated in the statements, we only need to prove (i). In this case, by  Theorem \ref{SecureTotalLexicographic} we know that  
$\gamma_{(1,1)}^s(P_n\circ H)=\gamma_{(1,1,0)}^s(P_n)$. Now, by Theorem \ref{PreliminaryBounds-w-Secure} (ii) and Corollary \ref{CorollarySecureSumaUno} we deduce that $\gamma_{(2,1,0)}(P_n)\le \gamma_{(1,1,0)}^s(P_n)\le \gamma_{(2,2,0)}(P_n)$. Moreover, as shown in \cite{w-domination}, $\gamma_{(2,1,0)}(P_n)=\gamma_{(2,2,0)}(P_n)=2\left\lceil \frac{n}{3}\right\rceil$, which completes the proof.
\end{proof}

By the result above and Theorem \ref{SecureTotalLexicographic} we deduce the following result.

\begin{proposition}\label{(110)s-caminos}
For any integer $n\ge 4$,
$$\gamma_{(1,1,0)}^s(P_n)=2\left\lceil \frac{n}{3}\right\rceil.$$
\end{proposition}

The following result concerns the case when $G$ is a cycle.

\begin{theorem}\label{SecureTotalLexicographic-Cycles}
The following statements hold for any integer $n\ge 4$ and any nontrivial graph $H$.
\begin{enumerate}
\item[{\rm (i)}] If $\gamma_{(1,1)}^s(H)=2$, then $\gamma_{(1,1)}^s(C_n\circ H)=\gamma_{(1,1,0)}^s(C_n)=
\left\{ \begin{array}{ll}
\left\lceil \frac{2n}{3}\right\rceil & if \, n=4,7,\\[5pt]
2\left\lceil \frac{n}{3}\right\rceil & \text{otherwise}.
\end{array}\right. $
\\
\item[{\rm (ii)}] If $\gamma(H)=1$ and $\gamma_{(1,1)}^s(H)\ge 3$, then $\gamma_{(1,1)}^s(C_n\circ H)=\gamma_{(1,1,1)}^s(C_n)\stackrel{\mbox{\rm \cite{TWRDF2018}}}{=}\left\lceil \frac{5n}{7}\right\rceil.$ \\
\item[{\rm (iii)}] If $\gamma(H)=2<\gamma_{(1,1)}^s(H)$, then $$\gamma_{(1,1)}^s(C_n\circ H)=\gamma_{(2,2,1)}(C_n)\stackrel{\mbox{\rm \cite{w-domination}}}{=} \left\{ \begin{array}{ll}
             n-\lfloor\frac{n}{7}\rfloor+1 & \text{if } \,  n\equiv 1,2 \pmod 7,\\[5pt]
             n-\lfloor\frac{n}{7}\rfloor & \mbox{otherwise.}
                                  \end{array}\right.$$
\item[{\rm (iv)}] If $\gamma(H)\geq 3$, then $\gamma_{(1,1)}^s(C_n\circ H)=\gamma_{(2,2,2)}(C_n)\stackrel{\mbox{\rm \cite{w-domination}}}{=}n. $
\end{enumerate}
\end{theorem}

\begin{proof}
As indicated in the statements, we only need to prove (i). In this case, by Theorem \ref{SecureTotalLexicographic} we know that $\gamma_{(1,1)}^s(C_n\circ H)=\gamma_{(1,1,0)}^s(C_n)$. As shown in \cite{w-domination}, $\gamma_{(2,1,0)}(C_n)=\left\lceil \frac{2n}{3}\right\rceil$ and $\gamma_{(2,2,0)}(C_n)=2\left\lceil \frac{n}{3}\right\rceil$. Moreover, by Theorem \ref{PreliminaryBounds-w-Secure}  (ii) and Corollary \ref{CorollarySecureSumaUno} we deduce that $\gamma_{(2,1,0)}(C_n)\leq \gamma_{(1,1,0)}^s(C_n)\leq \gamma_{(2,2,0)}(C_n)$.  Therefore, $\gamma_{(1,1,0)}^s(C_n)\le 2\left\lceil \frac{n}{3}\right\rceil$ and, if $n\equiv 0,2\pmod 3$, then $\gamma_{(1,1,0)}^s(C_n)=2\left\lceil \frac{n}{3}\right\rceil$.

 From now on, we consider that $n\equiv 1\pmod 3$ with $n\geq 10$, as the cases $n=4$ and $n=7$ are very  easy to check. Let $f(V_0,V_1,V_2)$ be a $\gamma_{(1,1,0)}^s(C_n)$-function and let $V(C_n)=\{x_0,\ldots , x_{n-1}\}$, where consecutive vertices are adjacent and the addition of subscripts is taken modulo $n$. If there exists $x_i\in V(C_n)$ such that $f(x_i)=f(x_{i+1})=0$, then $\gamma_{(1,1,0)}^s(C_n)=\gamma_{(1,1,0)}^s(P_n)$ and we derived the result by Proposition \ref{(110)s-caminos}. Hence, we assume that for any $x_i\in V_0$ we have that $f(x_{i-1})>0$ and $f(x_{i+1})>0$. From this fact, and considering that $f$ is a secure $(1,1,0)$-dominating function, we deduce that  $f(\{x_j,x_{j+1}, x_{j+2}\})\geq 2$ for any $x_j\in V(C_n)$. Now, we consider the following two cases.

\vspace{.2cm}

\noindent
Case $1$. $V_2\neq \emptyset$. Without loss of generality, we suppose that $x_0\in V_2$ and let $S_i=\{x_{3i+1}, x_{3i+2}, x_{3i+3}\}$ for $i\in \{0, \ldots ,\frac{n-4}{3}\}$. Notice that $f(S_i)\geq 2$ for every  $i\in \{0, \ldots ,\frac{n-4}{3}\}$, which implies that
$$\gamma_{(1,1,0)}^s(C_n)=\omega(f)\geq \sum_{i=0}^{\frac{n-4}{3}}f(S_i)+f(x_0)\geq \frac{2(n-1)}{3}+2=2\left\lceil \frac{n}{3}\right\rceil.$$
This implies that $\gamma_{(1,1,0)}^s(C_n)=2\left\lceil \frac{n}{3}\right\rceil.$

\vspace{.2cm}

\noindent
Case $2$. $V_2= \emptyset$. In this case, if there exist four consecutive vertices with weight one, namely without loss of generality $x_1,x_2,x_3,x_4$,  then we deduce that
$$\gamma_{(1,1,0)}^s(C_n)=\omega(f)\geq \sum_{i=5}^{n}f(x_i)+f(\{x_1,x_2,x_3,x_4\})\geq \frac{2(n-4)}{3}+4=2\left\lceil \frac{n}{3}\right\rceil.$$

From now on we assume that for any group of four consecutive vertices, at least one vertex has weight equal to zero.
Moreover, notice that $|V_0|\geq 2$ as $n\geq 10$ and $\gamma_{(1,1,0)}^s(C_n)\le 2\left\lceil \frac{n}{3}\right\rceil.$
 Without loss of generality, let $x_0\in V_0$. and suppose that $x_{1}\in M_f(x_0)$. This implies that $f(x_1)=f(x_2)=f(x_3)=1$ and $f(x_4)=0$, and so $f(x_5)=f(x_6)=1$.

If $f(x_7)=1$, then we have that
$$\gamma_{(1,1,0)}^s(C_n)=\omega(f)\geq \sum_{i=8}^{n}f(x_i)+f(\{x_1,x_2, \ldots ,x_7\})\geq \frac{2(n-7)}{3}+6=2\left\lceil \frac{n}{3}\right\rceil,$$

Now, assume that $f(x_7)=0$. Thus, $n\geq 13$ and $x_8\in M_f(x_7)$, which implies that 
$f(x_8)=f(x_9)=f(x_{10})=1$. Hence,
$$\gamma_{(1,1,0)}^s(C_n)=\omega(f)\geq \sum_{i=11}^{n}f(x_i)+f(\{x_1,x_2,\ldots , x_{10}\})\geq \frac{2(n-10)}{3}+8=2\left\lceil \frac{n}{3}\right\rceil,$$ 
which completes the proof.
\end{proof}

\end{document}